\documentclass{elsarticle}

\usepackage[all,cmtip]{xy}
\usepackage{array}
\usepackage{amsmath, amssymb}
\usepackage{graphics}
\usepackage{caption}
\usepackage{subcaption}
\usepackage{amsthm}
\usepackage{algorithm}
\usepackage{algorithmicx}
\usepackage{booktabs}
\usepackage{cancel}
\usepackage{caption}
\usepackage{subcaption}

\usepackage{algcompatible}
\usepackage[usenames,dvipsnames]{color}

\newtheorem{theorem}{{\bf Theorem}}
\newtheorem{remark}{{\bf Remark}}

\newtheorem{corollary}[theorem]{{\bf Corollary}}
\newtheorem{proposition}[theorem]{{\bf Proposition}}
\newtheorem{lemma}[theorem]{{\bf Lemma}}

\def\bfx{\boldsymbol{x}}
\def\bfy{\boldsymbol{y}}

\def\bfu{\boldsymbol{u}}
\def\bfv{\boldsymbol{v}}
\def\bfw{\boldsymbol{w}}

\def\bfN{\boldsymbol{N}}

\def\bfk{\boldsymbol{k}}
\def\bfI{\boldsymbol{I}}

\begin{document}
\begin{frontmatter}
%
% --- Author Metadata here ---
%\conferenceinfo{ISAAC}{'15 Bath, UK}
%\CopyrightYear{2007} % Allows default copyright year (20XX) to be over-ridden - IF NEED BE.
%\crdata{0-12345-67-8/90/01}  % Allows default copyright data (0-89791-88-6/97/05) to be over-ridden - IF NEED BE.
% --- End of Author Metadata ---

\title{On the computation of the straight lines contained in a rational surface.}

\author[a]{Juan Gerardo Alc\'azar\fnref{proy,proy2}}
\ead{juange.alcazar@uah.es}
\author[b]{Jorge Caravantes\fnref{proy}}
\ead{jcaravan@mat.ucm.es}

\address[a]{Departamento de F\'{\i}sica y Matem\'aticas, Universidad de Alcal\'a,
E-28871 Madrid, Spain}
\address[b]{Departamento de \'Algebra, Universidad Complutense de Madrid, E-28040 Madrid, Spain}
%\address[c]{Departamento de Matem\'atica Aplicada, Universidad de Murcia,  E-30100 Murcia, Spain} 

\fntext[proy]{Partially supported by the Spanish Ministerio de Econom\'{\i}a y Competitividad and by the European Regional Development Fund (ERDF), under the project  MTM2014-54141-P.}

\fntext[proy2]{Member of the Research Group {\sc asynacs} (Ref. {\sc ccee2011/r34}) }

\begin{abstract}
In this paper we present an algorithm to compute the (real and complex) straight lines contained in a rational surface, defined by a rational parameterization. The algorithm relies on the well-known theorem of Differential Geometry that characterizes real straight lines contained in a surface as curves that are simultaneously asymptotic lines, and geodesics. We also report on an implementation carried out in Maple 18, and we compare the behavior of our algorithm with two brute-force approaches.
\end{abstract}

\end{frontmatter}

% A category with the (minimum) three required fields
%\category{14Q10}{Computational aspects in algebraic geometry}{Surfaces}
%A category including the fourth, optional field follows...
%\category{D.2.8}{Software Engineering}{Metrics}[complexity measures, performance measures]

%\terms{Theory}

%\keywords{Algebraic Surfaces, Straight lines contained in Surfaces, Ruled surfaces, Cubic surfaces, Algorithms.}

\section{Introduction}

Straight lines are certainly notable curves in an algebraic surface. Probably the most famous result on algebraic surfaces containing straight lines is related to cubic surfaces: G. Salmon \cite{Salmon}, after correspondence with A. Cayley, proved that projective smooth cubic surfaces contain exactly 27 (projective, complex and real) straight lines, some of them at infinity. Schl\"affi \cite{Sch} proved later that the number of real straight lines must be 3, 7, 15 or 27. If the cubic is singular \cite{Bajaj2}, the number of straight lines goes down to 21.

Projective nonsingular cubic surfaces happen to be rational surfaces. If a parameterization of a rational cubic surface is known, one can compute the straight lines contained in the surface from the base points of the parameterization \cite{Bajaj2,Berry}. However, unlike cubics, surfaces of degree higher than 3 do not necessarily contain straight lines (see Theorem 1.27 in \cite{Shaf}). Furthermore, in the affirmative case, up to our knowledge there is no known algorithm other than a brute-force approach to find them. Brute-force algorithms to find straight lines proceed by translating the problem into polynomial system solving; this is certainly feasible for problems of small size, but turns increasingly non-feasible as the problem grows in size.

Computing the straight lines in a surface can be interesting on its own right, but it provides, additionally, useful information on the surface. Knowing the straight lines contained in a surface helps to find the symmetry center, symmetry planes and symmetry axes, if any, of the surface, since any of these symmetries maps the straight lines onto each other. Also, this information can help to identify whether or not two given surfaces are \emph{similar}, i.e. equal up to position and scaling, since two similar surfaces contain the same number of straight lines, that must be mapped to each other by the similarity. Additionally, if the surface contains some real line, then the surface is non-compact. 

Moreover, in the case of cubic surfaces Mordell proved (see Chapter 11 in \cite{Mordell}) that all rational points on a cubic surface can be found if the cubic surface contains two straight lines whose equations are defined by conjugate numbers of a quadratic field, and in particular by rational numbers. On the other hand, straight lines contained in certain surfaces are essential to describe the geometry of the surface, and therefore to accurately render it. For instance, Whitney's umbrella contains a straight line, the handle of the umbrella, sometimes missed in certain visualization packages. Steiner's surface contains three (possibly complex) double lines, intersecting at a triple singularity.

From the point of view of applications, in general a straight line contained in a curved artifact is a privileged curve that can be used for different purposes, either practical or ornamental. In particular straight lines contained in a surface are useful in Architecture \cite{Pottmann}, since such lines provide natural locations for beams supporting the structure. In fact, {\it minimal surfaces}, surfaces with the property of spanning a given boundary with minimum area, are used in Architecture and Mechanical Design \cite{Emmer, Wallner,Yoo}; and as discovered by Schwarz \cite{Dierkes}, any straight line contained in a minimal surface is an axis of symmetry of the surface.

In this paper we approach the problem of determining the straight lines contained in a surface defined by a rational parameterization of any degree. In order to do this, we exploit the well-known result in Differential Geometry that characterizes real non-singular straight lines contained in a surface, as curves that are simultaneously asymptotic lines, and geodesics. This characterization provides {\it differential} conditions to find the straight lines contained in the surface, that we transform into {\it algebraic conditions}; this way, we can take advantage of classical methods in polynomial algebra to solve the problem. Other special straight lines, in particular the ones contained in the singular part of the parameterization, can also be found. Additionally, the same strategy also allows to compute the complex straight lines contained in the surface.

The structure of the paper is the following. In Section 2 we provide some known and preliminary results. In Section 3 we present the algorithm. We prove that the algorithm {works properly} in Section 4. Section 5 reports on practical results, timings, etc. of an implementation of the algorithm in the computer algebra system Maple 18, that can be downloaded from \cite{Jorge}, as well as on the behavior of two brute-force approaches to the problem. Our conclusions are presented in Section 6. In Appendix I, we list the parametrizations used in the examples of Section 5. 

In the paper, sometimes we will use the term ``line" to refer to a curve contained in the surface we are analyzing. Hence, not every ``line" is a ``straight line", although the converse statement is, obviously, true.

\section{Notation, terminology, hypotheses and known results.} \label{sec-prelim.}

In this section we fix the notation, terminology and hypotheses to be used throughout this paper; we refer the interested reader to \cite{Docarmo, Patrikalakis, Struik} for further reading on the concepts and results mentioned here, mostly coming from the field of Differential Geometry. 

In this paper we work with a rational surface $S\subset \mathbb{R}^3$, not a plane, parameterized by 
\[\bfx(t,s)=(x(t,s),y(t,s),z(t,s)),\]
where $x(t,s)$, $y(t,s)$, $z(t,s)$ are real rational functions. We assume that this parameterization is {\it proper}, i.e. injective for almost all points of $S$, or equivalently that the mapping 
\[
\begin{array}{ccc}
\mathbb{ C}^2 & \to & S\\
(t,s) & \to & \bfx(t,s)
\end{array}
\]
is birational. One can check properness by using the algorithms in \cite{PDSS02, PDS04}; for reparameterization questions one can see \cite{PD06, PD13}. We will consider the Euclidean space $\mathbb{ R}^3$ furnished with the usual dot product $\langle\mbox{ , }\rangle$, and the usual Euclidean norm $\Vert \cdot \Vert$. Moreover, we use the notation $\bullet_t$ (resp. $\bullet_s$) for the partial derivative of $\bullet$ with respect to $t$ (resp. $s$); similarly $\bullet_{tt}$, $\bullet_{ts}$, $\bullet_{ss}$ represent second partial derivatives of $\bullet$. 

We say that $\bfx(t,s)$ is {\it regular} at a point $P_0=\bfx(t_0,s_0)$, if $({\bfx}_t\times {\bfx}_s)(t_0,s_0)\neq 0$. At such a point we can define the {\it normal vector} $\bfN$,
\begin{equation}\label{eq-normal}
\bfN=\frac{{\bfx}_t\times {\bfx}_s}{\Vert {\bfx}_t\times {\bfx}_s \Vert}.
\end{equation}
If $({\bfx}_t\times {\bfx}_s)(t_0,s_0)=0$, in which case $\bfN$ is not defined, we say that $\bfx(t_0,s_0)$ is a {\it singular point of the parameterization $\bfx(t,s)$}; we denote by $\mbox{Sing}_{\bfx}$ the set consisting of all the singular points of $\bfx(t,s)$. Furthermore, we represent the matrices defining the \emph{first fundamental form} and \emph{second fundamental form} of $S$ by $I,II$, respectively; recall that
\begin{equation}\label{forms}
\hspace{-0.4 cm}\begin{array}{cc}
\begin{array}{cc}
I=\begin{bmatrix}
 E & F\\
 F & G
\end{bmatrix}=\begin{bmatrix}
 \langle\bfx_t, \bfx_t \rangle & \langle\bfx_t,\bfx_s\rangle\\
 \langle\bfx_s, \bfx_t \rangle & \langle\bfx_s, \bfx_s\rangle
\end{bmatrix},
\end{array} & 
\begin{array}{cc}
II=\begin{bmatrix}
 e & f\\
 f & g
\end{bmatrix}=\begin{bmatrix}
\langle \bfx_{tt}, \bfN \rangle& \langle\bfx_{ts},\bfN\rangle\\
\langle \bfx_{st},\bfN \rangle& \langle\bfx_{ss}, \bfN\rangle
\end{bmatrix}.
\end{array}
\end{array}
\end{equation}

\noindent We also introduce the following notation:
\begin{equation}\label{star}
e^{\star}=\langle \bfx_{tt},\bfx_t\times \bfx_s\rangle, \hspace{0.4 cm}f^{\star}=\langle \bfx_{ts},\bfx_t\times \bfx_s\rangle, \hspace{0.4 cm}g^{\star}=\langle \bfx_{ss},\bfx_t\times \bfx_s\rangle.
\end{equation}
Notice that $e^{\star}$, $f^{\star}$, $g^{\star}$ result from multiplying $e,f,g$ by $\Vert {\bfx}_t\times {\bfx}_s\Vert$. Since $\bfx(t,s)$ is rational, $e^{\star}$, $f^{\star}$, $g^{\star}$ are rational functions. The following result is well-known. 

\begin{theorem}\label{known}
Let ${\mathcal C}\subset S$ be a real curve contained in a surface $S$ parameterized by $\bfx(t,s)$, and assume that ${\mathcal C}\not\subset \mbox{Sing}_{\bfx}$. Then ${\mathcal C}$ is a straight line iff it is simultaneously an asymptotic line, and a geodesic line. 
\end{theorem}

We recall here some of the notions involved in the theorem above for the convenience of the reader. Let ${\mathcal C}\subset S$ be a parameterized curve contained in $S$, and let $\{{\bf t}, {\bf n}, {\bf b}\}$ denote the \emph{Frenet frame} of ${\mathcal C}$. Then the \emph{curvature} vector $\bfk$ of ${\mathcal C}$ is defined as $\bfk=\dot{\bf t}$, where the dot means the derivative with respect to the \emph{arc-length} parameter. Furthermore, the \emph{normal curvature}, $k_n$, and \emph{geodesic curvature}, $k_g$, are defined as 
\begin{equation}\label{kn}
k_n=\left\langle\dot{\bf t}, {\bfN}\right\rangle,\mbox{ }k_g=\left\langle\dot{\bf t},{\bfN}\times {\bf t}\right\rangle.
\end{equation}
Additionally, the normal curvature vector, $\bfk_n$, and the geodesic curvature vector, $\bfk_g$, are defined as 
\begin{equation}\label{normalcurvat}
\bfk_n=k_n\cdot \bfN,\mbox{ }\bfk_g=k_g\cdot (\bfN \times {\bf t}),
\end{equation}
so that $\bfk=\bfk_n+\bfk_g$. Then we say that ${\mathcal C}$ is an \emph{asymptotic line} of ${\mathcal C}$ if $k_n=0$, and we say that ${\mathcal C}$ is a \emph{geodesic line} if $k_g=0$. 

For cardinality reasons and since $\bfx$ parametrizes a surface, if ${\mathcal L}\subset S$ is a straight line covered by the parameterization $\bfx$, i.e. such that almost all points of ${\mathcal L}$ are the image of some point $(t,s)$ in the parameter space, the Zariski closure of $\bfx^{-1}(\mathcal{L})$ has dimension 1. In other words, there is an algebraic curve $p(t,s)=0$ in the $(t,s)$ plane whose image under $\bfx$ is ${\mathcal L}$. Conversely, consider a curve ${\mathcal C}\subset S$, described as the set of points $\bfx(t,s)$ with $p(t,s)=0$, where $p(t,s)$ is polynomial and square-free. Since $p(t,s)$ does not have multiple components, there exists a point $(t_0,s_0)$ where $p(t_0,s_0)=0$ and some partial derivative $p_t,p_s$ is nonzero. Applying the Implicit Function Theorem at $(t_0,s_0)$, we have that ${\mathcal C}$ can, at least locally, be parameterized as either $\bfx(t,s(t))$, where $s(t)$ is implicitly defined by $p(t,s)=0$ in the vicinity of $(t_0,s_0)$, or $\bfx(t_0,s)$. In the first case, if ${\mathcal C}$, parameterized by $\bfx(t,s(t))$, is not contained in $\mbox{Sing}_{\bfx}$ and $s(t)$ is a real function, then ${\mathcal C}$ is an asymptotic line of $S$ iff $s(t)$ satifies (see \cite[Ch.2, 5-13]{Struik})

\begin{equation} \label{asint2}
e^{\star}+2f^{\star}\cdot \frac{ds}{dt}+g^{\star}\cdot \left(\frac{ds}{dt}\right)^2=0.
\end{equation}

\noindent Additionally, in the same conditions ${\mathcal C}$ is a geodesic line of $S$ iff $s(t)$ satifies (see \cite[Ch.4, 2-3a]{Struik})
\begin{equation} \label{geod2}
\displaystyle{\bfI\cdot\frac{d^2s}{dt^2}=\widehat{\Gamma}^1_{22}\cdot \left(\frac{ds}{dt}\right)^3+\left(2\widehat{\Gamma}^1_{12}-\widehat{\Gamma}^2_{22}\right)\left(\frac{ds}{dt}\right)^2+\left(\widehat{\Gamma}^1_{11}-2\widehat{\Gamma}^2_{12}\right)\frac{ds}{dt}-\widehat{\Gamma}^2_{11},}
\end{equation}
where $\widehat{\Gamma}^i_{jk}=\Gamma^i_{jk}\cdot \bfI$, with $\bfI=EG-F^2$ (the determinant of the first fundamental form), and where the $\Gamma^i_{jk}$ are the {\it Christoffel symbols} of $\bfx(t,s)$. One can find explicit formulae for the Christoffel symbols, for instance, in page 268 of \cite{Patrikalakis}. Since $\bfx(t,s)$ is rational, the $\widehat{\Gamma}^i_{jk}$ are rational functions of $t,s$. 

As for a curve $\bfx({\bf c},s)$, with ${\bf c}$ a constant, considering analogous formulae to \eqref{asint2} and \eqref{geod2} in terms of $\frac{dt}{ds}$ instead of $\frac{ds}{dt}$, one can see that {such a curve} is an asymptotic line iff $g^{\star}=0$, and is a geodesic line iff $\widehat{\Gamma}^1_{22}=0$. Therefore, we have the following result, which will be useful in Section \ref{sec-comput}; here we denote 

\begin{equation}\label{eta}
\eta(t,s)=\gcd\left(\mbox{num}\left(g^{\star}(t,s)\right),\mbox{num}\left(\widehat{\Gamma}^1_{22}(t,s)\right)\right),
\end{equation}
where $\mbox{num}(\bullet)$ represents the numerator of $\bullet$.

\begin{lemma} \label{aux-str}
The curve $\bfx({\bf c},s)$, where ${\bf c}\in \mathbb{ R}$, is a straight line iff $t-{\bf c}$ divides $\eta(t,s)$.
\end{lemma}

Finally, we recall that a surface $S$ is said to be {\it ruled} if at every point $P\in S$ there exists a (real or complex) straight line $L_P$ through $P$ contained in $S$.

\section{Computation of the straight lines.} \label{sec-comput}

%\subsection{Straight lines contained in the regular part of $S$.} \label{subsec-reg}

In this section we describe the algorithm for computing the straight lines contained in the surface $S$ that are covered by the parameterization $\bfx(t,s)$. Results on the computation of the curves contained in a rational surface not covered by the parameterization, or methods for surjectively reparameterizing rational surfaces, can be found in \cite{Bajaj, Chou} for special kinds of surfaces, and also in \cite{PSV, SSV}. Observe that in any case, the set of points not covered by the parameterization $\bfx(t,s)$ has at most dimension 1. The algorithm we present in this section is inspired in Theorem \ref{known}. In fact, the correctness of the algorithm, for the case of real straight lines, follows from Theorem \ref{known}. However, in Subsection \ref{subsec-complex} we justify that the same ideas lead to the computation of the complex straight lines, too. 

\vspace{2 mm}

\noindent \underline{Overview of the section:}

\vspace{2 mm}

First, in Subsection \ref{subsec-develop} we consider the computation of the straight lines not contained in $\mbox{Sing}_{\bfx}$, that can be locally parameterized as $\bfx(t,s(t))$, where $s(t)$ satisfies both Eq. \eqref{asint2} and \eqref{geod2} (see Section \ref{sec-prelim.}). We address this computation under the most general conditions, leaving special situations for a later subsection. The main idea is as follows: Eq. \eqref{asint2} and Eq. \eqref{geod2} provide conditions involving $t,s,\frac{ds}{dt},\frac{d^2s}{dt^2}$ for the function $s(t)$. Our strategy then is to move from these {\it differential} conditions to {\it algebraic} conditions. In order to do this, we replace $\frac{ds}{dt},\frac{d^2s}{dt^2}$ in Eq. \eqref{asint2} and Eq. \eqref{geod2} by two new variables $\omega:=\frac{ds}{dt}$, $\gamma:=\frac{d^2s}{dt^2}$, linked by the relationship $\gamma=\frac{d\omega}{dt}$. {An extra condition involving $\omega,\gamma$} is obtained by differentiating Eq. \eqref{asint2} with respect to $t$, and using that $\gamma=\frac{d\omega}{dt}$. Then we employ algebraic methods, namely resultants, to eliminate the variables $\omega,\gamma$, and we derive a polynomial condition $\xi(t,s)=0$. The curves, contained in $S$, parameterized as
\begin{equation}\label{forma}
\bfx(t,s),\mbox{ }\mu_i(t,s)=0,
\end{equation}
where $\mu_i(t,s)$ is an irreducible factor of $\xi(t,s)$, are then potential {\it candidates} for straight lines contained in $S$. However, not all these curves are straight lines. A criterion to distinguish those ones {which are straight lines from those ones which are not}, is given in Subsection \ref{checking}. 

In Subsection \ref{subsec-spec} we address the computation of the straight lines corresponding to some special situations, left aside in Subsection \ref{subsec-develop}. In particular, here we include straight lines contained in $\mbox{Sing}_{\bfx}$, as well as straight lines parameterized as $\bfx({\bf c},s)$, with ${\bf c}$ a constant. The whole algorithm is given in Subsection \ref{subsec-alg-rat}. Next, in Subsection \ref{subsec-complex} we prove why the algorithm provides not only the real, but also the complex straight lines contained in the surface. Finally, a detailed example is presented in Subsection \ref{detail-ex}.

\subsection{Development of the method.} \label{subsec-develop}

Let us replace $\omega:=\frac{ds}{dt}$ in Eq. \eqref{asint2}. This way we get 
\begin{equation}\label{as-rat}
e^{\star}+2f^{\star}\cdot \omega+g^{\star}\cdot \omega^2=0,
\end{equation}
where $e^{\star},f^{\star},g^{\star}$ are given in Eq. \eqref{star}, and are assumed to be evaluated at $(t,s(t))$, where the function $s=s(t)$ is unknown; for simplicity, however, in the rest of the section we will write $s$ instead of $s(t)$. Differentiating Eq. \eqref{as-rat} with respect to the variable $t$, and introducing the variable $\gamma:=\frac{d\omega}{dt}$, we get

\begin{equation} \label{diff1}
A(t,s,\omega)\cdot \gamma+B(t,s,\omega)=0,
\end{equation}
where 

\begin{equation} \label{EqA}
A(t,s,\omega)=2(f^{\star}+g^{\star}\omega)
\end{equation}
and 
\begin{equation}\label{EqB}
B(t,s,\omega)=g^{\star}_s\cdot \omega^3+(2f^{\star}_s+g^{\star}_t)\omega^2+(2f^{\star}_t+e^{\star}_s)\cdot \omega+e^{\star}_t.
\end{equation}
%The straight lines consisting of points of $S$ where $f^{\star}+g^{\star}\omega=0$, if any, must be determined separately, using an analogous strategy. 

Moreover, if ${\mathcal C}$ is also a geodesic of $S$ then from Eq. \eqref{geod2} and using the variables $\omega:=\frac{ds}{dt}$, $\gamma:=\frac{d\omega}{dt}=\frac{d^2s}{dt^2}$, we have 
\begin{equation} \label{diff2}
\bfI\cdot \gamma=\widehat{\Gamma}^1_{22}\omega^3+(2\widehat{\Gamma}^1_{12}-\widehat{\Gamma}^2_{22})\omega^2+(\widehat{\Gamma}^1_{11}-2\widehat{\Gamma}^2_{12})\omega-\widehat{\Gamma}^2_{11}.
\end{equation}

If $f^{\star}$ and $g^{\star}$ are not both identically $0$ (the special case when both $f^{\star},g^{\star}$ are zero will be treated later), $f^{\star}+g^{\star}\omega$, seen as a polynomial in $\omega$, is not identically $0$ either. Then from Eq. \eqref{diff1} we can write $\gamma=-B(t,s,\omega)/A(t,s,\omega)$. Substituting this expression into Eq. \eqref{diff2} and clearing denominators we get a polynomial relationship between $t,s,\omega$,   
\begin{equation} \label{diff3}
\tilde{N}(t,s,\omega)=0.
\end{equation}

\noindent Let $\tilde{M}(t,s,\omega)$ be the polynomial obtained by clearing denominators in Eq. \eqref{as-rat}, and let $M(t,s,\omega)$, $N(t,s,\omega)$ be the primitive parts of $\tilde{M}(t,s,\omega)$ and $\tilde{N}(t,s,\omega)$ with respect to $\omega$. Finally, let $\tilde{\xi}(t,s)$ be the resultant of $M(t,s,\omega)$ and $N(t,s,\omega)$ with respect to $\omega$, and let $\xi(t,s)$ be the result of removing from $\tilde{\xi}(t,s)$ the denominators of the components of $\bfx(t,s)$ and the denominators of the rational functions in Eq. \eqref{as-rat}, \eqref{diff1}, \eqref{diff2}, if any. 

Notice that $\xi(t,s)=0$ implicitly defines $s=s(t)$. Therefore, the components of the space curve defined by $\bfx(t,s)$, where $\xi(t,s)=0$, are potential candidates for straight lines contained in $S$. Each of these components can be described as 
\begin{equation}\label{components-curve}
\bfx(t,s),\mbox{ }\mu_i(t,s)=0,
\end{equation}
where $\mu_i(t,s)$ is an irreducible factor of $\xi(t,s)$. If we only want to find the real straight lines contained in $S$ we need irreducible factors over the reals (not only over ${\Bbb Q}$); if we want the real and complex straight lines, we need the irreducible factors over the complex, i.e. an absolute factorization of the polynomial $\xi(t,s)$.

\noindent\begin{remark} \label{components} There is a number of papers where the problem of factoring over the reals is addressed, see for instance \cite{Che,CGKW,Gao,Kal85,Kal08}. In our case, we used the command {\tt AFactors} of Maple 18, which works finely for polynomials of moderate and medium degrees. The same command also allows to factor over the complex numbers. 
\end{remark}

\subsection{Checking candidates.} \label{checking}

Not every every curve \eqref{components-curve} gives rise to a straight line contained in $S$. This is not contradictory, since in the process to compute the polynomial $\xi(t,s)$ we treat $\omega$ as a variable independent of $t,s$, without taking into account the (differential) relationship $\omega=\frac{ds}{dt}$. Therefore, we need a criterion to distinguish those real factors of $\xi(t,s)$ which give rise to straight lines contained in $S$, from those which do not. In order to do this, we use the following result. 

\noindent\begin{proposition} \label{check-line}
Let ${\mathcal C}$ be the set of points of $\mathbb{C}^3$ parameterized by $\bfx(t,s)$ with $\alpha(t,s)=0$, where $\alpha(t,s)$ is irreducible over the complex and depends explicitly on $s$. Then ${\mathcal C}$ corresponds to a straight line if and only if there exists a rational function $\lambda(t,s)$ and a constant vector $\bfu$ such that the following equality 
\begin{equation}\label{condi}
(x_t\alpha_s-x_s\alpha_t,y_t\alpha_s-y_s\alpha_t,z_t\alpha_s-z_s\alpha_t)=\lambda(t,s)\cdot \bfu
\end{equation}
holds modulo $\alpha(t,s)$.
\end{proposition}

\noindent\begin{proof} If $\alpha(t,s)$ depends on $s$, then from the Implicit Function Theorem $\alpha(t,s)=0$ (locally) defines $s=s(t)$. Furthermore, $\frac{ds}{dt}=-\alpha_t/\alpha_s$. Using this, the tangent vector to a point of ${\mathcal C}$ is parallel to
\[(x_t\alpha_s-x_s\alpha_t,y_t\alpha_s-y_s\alpha_t,z_t\alpha_s-z_s\alpha_t).\]Then the condition in the statement of the lemma is equivalent to all the tangents to ${\mathcal C}$ being parallel to a same direction. 
\end{proof}

In order to check condition \eqref{condi}, one proceeds in the following way for each irreducible factor $\mu_i(t,s)$ of $\xi(t,s)$ (in general, over ${\Bbb C}$), not a common factor\footnote{In such a case, the space curve defined by $\bfx(t,s)$ with $\mu_i(t,s)=0$ degenerates into a point.} of the components of the left hand-side of \eqref{condi}, taking $\alpha(t,s)\equiv \mu_i(t,s)$ for each value of $i$:

\begin{itemize}
\item [1.] Let $t=a\in \mathbb{ Z}$. Then $(a,b)$, where $m(b)=\alpha(a,b)=0$, is a point of the curve $\alpha(t,s)=0$.
\item [2.] Let $\bfw(t,s)=(x_t\alpha_s-x_s\alpha_t,y_t\alpha_s-y_s\alpha_t,z_t\alpha_s-z_s\alpha_t)$, and let $\bfw_0$ be the result of evaluating $\bfw(t,s)$ at the point $t=a$, $s=b$. If $\bfw_0$ is zero, we choose a different $(a,b)$. 
\item [3.] ${\mathcal C}$ is a straight line iff all the components of $\bfw(t,s)\times \bfw_0$ are divisible by $\alpha(t,s)$. Furthermore, in the affirmative case ${\mathcal C}$ is parallel to $\bfw_0$.
\end{itemize}

\subsection{Remaining straight lines.} \label{subsec-spec}

We need to pursue now some additional straight lines that may have been missed in the process described in Subsection \ref{subsec-develop}, including the straight lines contained in $\mbox{Sing}_{\bfx}$. Assume first, as we did in Subsection \ref{subsec-develop}, that $f^{\star},g^{\star}$ are not identically zero. Then we need to examine the lines parameterized by $\bfx(t,s)$ with $\delta_j(t,s)=0$, $j=1,\ldots,4$, where: 
\begin{itemize}
\item [(1)] $\delta_1(t,s)$ is the $\gcd$ of the coefficients of the numerator of $A(t,s,\omega)$, seen as a polynomial of degree 1 in $\omega$. 
\item [(2)] $\delta_2(t,s)$ is the product of the contents of $M(t,s,\omega)$ and $N(t,s,\omega)$ with respect to $\omega$; notice that these contents were eliminated when we moved from $\tilde{M},\tilde{N}$ to $M,N$. 
\item [(3)] In order to compute the straight lines, if any, consisting of points $\bfx(t,s)$ where $f^{\star}+g^{\star}\omega=0$, we observe that in this case we have 
\[\omega=-f^{\star}/g^{\star}.\]
Substituting this expression in Eq. \eqref{as-rat} and clearing the denominator yields the condition
\[\delta_3(t,s)=e^{\star} g^{\star}-\left(f^{\star}\right)^2=0.
\]
Notice that $\delta_3(t,s)$ is the determinant of the second fundamental form multiplied by $\Vert \bfN\Vert^2$ (see Eq. \eqref{forms}). 
\item [(4)] {\it Straight lines contained in $\mbox{Sing}_{\bfx}$}: we define $\delta_4(t,s)$ as the numerator of $\Vert \bfx_t\times \bfx_s\Vert^2$. Notice that at the points where $\bfx(t,s)$ is singular, we have $\Vert \bfx_t\times \bfx_s\Vert^2=0$ .
\end{itemize}

For each curve $\bfx(t,s)$, with $\delta_i(t,s)=0$, $i=1,\ldots,4$, we can also use Proposition \ref{check-line} in order to check whether or not the curve corresponds to a straight line. Additionally, in the case when $f^{\star},g^{\star}$ are not both identically zero we also need to compute the straight lines parameterized as $\bfx({\bf c},s)$, where ${\bf c}$ is a constant. These straight lines are found by applying Lemma \ref{aux-str}.

Finally, we address the special case when $f^{\star},g^{\star}$ are both identically 0. In this situation $e^{\star}$ cannot be identically $0$, because otherwise $S$ is a plane (see pg. 147 of \cite{Docarmo}), which is excluded by hypothesis. Hence \eqref{as-rat} reduces to $e^{\star}(t,s)=0$. If $e^{\star}(t,s)=e^{\star}(t)$ then we just need to look for lines that can be parameterized as $\bfx({\bf c},s)$ with ${\bf c}$ a root of $e^{\star}(t)$, which can be done by applying Lemma \ref{aux-str}. Otherwise, $e^{\star}(t,s)=0$ implicitly defines $s=s(t)$; differentiating $e^{\star}(t,s)=0$ twice with respect to $t$, and using the variables $\omega:=\frac{ds}{dt}$, $\gamma:=\frac{d\omega}{dt}$ as before, we get an expression like Eq. \eqref{diff1}, with 
\begin{equation}\label{newAB}
A(t,s,\omega)=e^{\star}_s,\mbox{ }B(t,s,\omega)=e^{\star}_{ss}\cdot \omega^2+2e^{\star}_{ts}\cdot \omega+e^{\star}_{tt}.
\end{equation}

\noindent Now writing $\gamma=-B(t,s,\omega)/A(t,s,\omega)$, we can proceed as in Subsection \ref{subsec-develop} to find an expression like Eq. \eqref{diff3}. From this moment on the situation is analogous to that in Subsection \ref{subsec-develop}, and a polynomial $\xi(t,s)$ is computed. In this case we also have to examine the lines parameterized as $\bfx(t,s)$, $\delta_i(t,s)=0$, where $i=1,\ldots,4$ and the polynomials $\delta_1(t,s)$, $\delta_2(t,s)$ and $\delta_4(t,s)$ are computed as before. No polynomial analogous to $\delta_3(t,s)$ appears now, though, so in this case we define $\delta_3(t,s):=1$.

\subsection{The whole algorithm.} \label{subsec-alg-rat}

From the ideas and results in the previous subsections we can derive the following algorithm, Algorithm {\tt STLines}, to compute the straight lines contained in $S$. The algorithm requires that the surface $S$ defined by $\bfx(t,s)$ is not ruled. Ruled surfaces of degree higher than two can be characterized by means of the {\it Pick's invariant} of the surface, which is defined as $J=K-H$, where $K$ is the {\it Gauss' curvature} of the surface and $H$ is the {\it mean curvature} of the surface. Ruled surfaces that are not quadrics are exactly those ones with vanishing $J$ (see pages 89, 90 of \cite{Nomizu}). On the other hand, every quadric contains infinitely many complex lines, so every quadric can be regarded as a ruled surface. However, in Section \ref{term} we will see that in fact we do not need to check whether or not $S$ is ruled in advance, since ruled surfaces can be recognized while we are running the algorithm. 

Notice that the algorithm has two loops, in steps 14 and 19. The number of iterations in each of these loops equals the number of factors of the polynomials $\mu(t,s)$ and $\eta(t,s)$, computed in the steps 13 and 18, respectively. Therefore, in order to show that the algorithm works properly we need to show that under the considered hypotheses $\mu(t,s),\eta(t,s)$ are not identically zero; otherwise the number of factors, and therefore of iterations in the loops of either step 14 or step 19, would be infinite. We will prove this in Section \ref{term}.

\begin{algorithm}
\begin{algorithmic}[1]
\REQUIRE A proper parameterization $\bfx(t,s)$ of an algebraic surface $S$, which is not ruled.  
\ENSURE The real and complex straight lines contained in $S$ that are covered by the parameterization $\bfx(t,s)$.
\STATE \emph{{\bf Part I:} Straight lines \emph{not} of the type $\bfx({\bf c},s)$, where ${\bf c}$ is a constant.}
\STATE{find the numerator $\tilde{M}(t,s,\omega)$ of the left hand-side of \eqref{as-rat}.  }
\IF{$f^{\star}, g^{\star}$ are not both identically zero}
\STATE{compute $A(t,s,\omega),B(t,s,\omega)$ using Eq. \eqref{EqA} and Eq. \eqref{EqB}.}
\ELSE
\STATE{compute $A(t,s,\omega),B(t,s,\omega)$ using Eq. \eqref{newAB}.}
\ENDIF
\STATE{substitute $\gamma=-B(t,s,\omega)/A(t,s,\omega)$ into Eq. \eqref{diff2} and clear denominators to get $\tilde{N}(t,s,\omega)$. }
\STATE{compute the primitive parts $M(t,s,\omega)$, $N(t,s,\omega)$ of $\tilde{M}(t,s,\omega)$ and $\tilde{N}(t,s,\omega)$ with respect to $\omega$.}
\STATE{compute $\tilde{\xi}(t,s)=\mbox{Res}_{\omega}(M(t,s,\omega),N(t,s,\omega))$.}
\STATE{let $\xi(t,s)$ be the result of removing from $\tilde{\xi}(t,s)$ the denominators of the components of $\bfx(t,s)$, and the denominators of the rational functions in Eqs. \eqref{as-rat}, \eqref{diff1}, \eqref{diff2}, if any. }
\STATE{find the polynomials $\delta_1(t,s)$, $\delta_2(t,s)$, $\delta_3(t,s)$, $\delta_4(t,s)$ defined in Subsec. \ref{subsec-spec}.}
\STATE{compute $\mu(t,s)$, the square-free part of $\mu^{\star}(t,s)=\xi(t,s)\cdot \delta_1(t,s)\cdot \delta_2(t,s)\cdot \delta_3(t,s)\cdot \delta_4(t,s)$.}
\FOR{each irreducible component of $\mu(t,s)=0$ over the complex,}
\STATE{use Proposition \ref{check-line} to check whether or not it corresponds to a straight line contained in $S$.}
\ENDFOR
\STATE{\emph{{\bf Part II:} Straight lines of the type $\bfx({\bf c},s)$, where ${\bf c}$ is a constant.}}
\STATE{let $\eta(t,s)=\gcd\left(\mbox{num}(g^{\star}(t,s)),\mbox{num}\left(\widehat{\Gamma}^1_{22}(t,s)\right)\right)$. }
\FOR{each factor $t-{\bf c}_i$ of $\eta(t,s)$, only depending on $t$}
\STATE{compute the line $\bfx({\bf c}_i,s)$, $i=1,\ldots,n$.}
\ENDFOR
\STATE{{\bf return} the list of straight lines computed in Part I, Part II, or the message {\tt no straight lines found}.}
\end{algorithmic}
\caption{{\tt STLines}}
\end{algorithm}

\subsection{Why we also get the complex straight lines.}\label{subsec-complex}

In this subsection we show that, whenever we compute an absolute factorization of the polynomials $\mu(t,s)$ and $\eta(t,s)$ in steps 13 and 18 of Algorithm {\tt STLines}, i.e. a factorization over the complex numbers, Algorithm {\tt STLines} also provides the complex straight lines contained in $S$. The fact that Algorithm {\tt STLines} also provides the complex straight lines contained in $S$ is not necessarily obvious: although these lines can be parameterized as $\bfx(t,s(t))$ or $\bfx({\bf c},s)$, in this case $s(t)$, ${\bf c}$ are complex, while the results in Section \ref{sec-prelim.}, and therefore Theorem \ref{known} too, assume that $s(t)$ or $t(s)$ are real functions. Therefore, we need to see that Eq. \eqref{asint2} and Eq. \eqref{geod2} in Section \ref{sec-prelim.} also hold when $s(t),{\bf c}$ are complex. 

In order to do this, we need to look closer at the notions of {\it normal curvature} $k_n$, {\it normal curvature vector} $\bfk_n$, {\it geodesic curvature} $k_g$,  and {\it geodesic curvature vector} $\bfk_g$, that we recalled in Section \ref{sec-prelim.}. For the curves $\bfx(t,s(t))$ or $\bfx(t(s),s)$, with $s(t),t(s)$ real functions, the conditions $k_n=0$ and $k_g=0$ are equivalent to Eq. \eqref{asint2} and Eq. \eqref{geod2}. However, when $s(t),t(s)$ are complex, $\bfk_n,\bfk_g$ are not necessarily well-defined. The reason is that $\langle\mbox{ , }\rangle$ is not an inner product in $\mathbb{C}^3$, since the positive-definiteness property does not hold anymore. Because of this, the normal vector $\bfN$ is not necessarily well-defined even though $\bfx(t,s)$, and therefore $\bfx_t, \bfx_s$, are well-defined, because $\Vert \bfx_t\times \bfx_s\Vert$ can vanish although $\bfx_t\times \bfx_s$ is nonzero. Also, if we have a space curve parameterized by $\bfy(t)$, it can happen that $\bfy'(t)\neq 0$ but $\Vert \bfy'(t)\Vert =0$, in which case the unitary tangent ${\bf t}$ is not well-defined either. 

In order to see that complex straight lines satisfy Eq. \eqref{asint2} and Eq. \eqref{geod2} too, we need the following lemmata, the first of which can be easily proven. Here, ${\mathcal L}(\bfu,\bfv)$ represents the linear variety spanned by the vectors $\bfu,\bfv$.

\begin{lemma} \label{eq1}
Let ${\mathcal C}\subset S$ be parameterized by $\bfy(t)=\bfx(t,s(t))$ for some real function $s(t)$.
\begin{itemize}
\item [(1)] The condition $k_n=0$ is equivalent to $\bfy''\in {\mathcal L}(\bfx_t,\bfx_s)$, where $\bfy''$ is evaluated at $t$, and $\bfx_t,\bfx_s$ at $(t,s(t))$. 
\item [(2)] The condition $k_g=0$ is equivalent to $\langle\bfy'',(\bfx_t\times \bfx_s)\times \bfy'\rangle=0$, where $\bfy',\bfy''$ are evaluated at $t$, and $\bfx_t,\bfx_s$ at $(t,s(t))$. 
\end{itemize}
\end{lemma} 

%\begin{proof} (1) ${\bf t}$ is parallel to $\bfy'$, and $\frac{d{\bf t}}{d\bfs}\in{\mathcal L}(\bfy',\bfy'')$. Since $\langle \bfy',\bfN\rangle=0$, from the definition of $k_n$ %given in Eq. \eqref{kn}, the condition $k_n=0$ is equivalent to $\langle\bfy'', \bfN\rangle=0$. In turn, this condition is equivalent to $\bfy''\in {\mathcal L}(\bfx_t,\bfx_s)$. %(2) It follows from the definition of $k_g$ given in Eq. \eqref{kn}, taking into account that ${\bf t}$ is parallel to $\bfy'$, $\frac{d{\bf t}}{d\bfs}\in{\mathcal %L}(\bfy',\bfy'')$, and $\bfN$ is parallel to $\bfx_t\times \bfx_s$. 
%\end{proof}

\begin{lemma} \label{rectacom}
Let $\bfy(t)=\bfx(t,s(t))$ be a complex straight line (i.e. $s(t)$ is a complex function such that the imaginary part of $\bfy(t)$ is nonzero) contained in $S$. Then, $\bfy(t)$ satisfies that: (1) $\bfy''\in {\mathcal L}(\bfx_t,\bfx_s)$; (2) $\langle\bfy'',(\bfx_t\times \bfx_s)\times \bfy'\rangle=0$.
\end{lemma}

\begin{proof} (1) $\bfy(t)=\bfx(t,s(t))$ is a straight line iff $\bfy'(t),\bfy''(t)$ are linearly dependent for all $t$. Since $\bfy'\in {\mathcal L}(\bfx_t,\bfx_s)$, it follows that $\bfy''\in {\mathcal L}(\bfx_t,\bfx_s)$ too. Therefore, condition (1) holds. (2) The triple product $\langle\bfy'',(\bfx_t\times \bfx_s)\times \bfy'\rangle$ can be written as a determinant, where the first and last rows correspond to the coordinates of $\bfy''$ and $\bfy'$ respectively. Since these rows are proportional, the value of the determinant is zero. 
\end{proof}

\begin{remark} One can prove a completely analogous lemma for curves $\bfx(t(s),s)\subset S$, with $t(s)$ a complex function. In particular, curves $\bfx({\bf c},s)$ with ${\bf c}\in \mathbb{C}$ are of this type. From here one can conclude that Lemma \ref{aux-str} also holds when ${\bf c}\in \mathbb{C}$.
\end{remark}

The conditions of Lemma \ref{rectacom} are exactly the conditions appearing in the statements (1) and (2) of Lemma \ref{eq1}. But these conditions imply Eq. \eqref{asint2} and Eq. \eqref{geod2}. 

\begin{corollary}
If $\bfx(t,s(t))$ is a real or complex straight line contained in $S$, then $s(t)$ satisfies Eq. \eqref{asint2} and Eq. \eqref{geod2}. 
\end{corollary}

Interestingly enough, and unlike the real case, Eq. \eqref{asint2} and Eq. \eqref{geod2} are necessary conditions for non-singular complex straight lines, but they are not sufficient. Consider for instance the surface
$S$ parameterized by 
\[
\bfx(t,s)=\left(\frac{1}{2}t(s^3+3s)+\frac{1}{2}s^2,t^2,t+s\right),
\]
and let $s(t)=i$, where $i^2=-1$. Then $\bfy(t)=\bfx(t,i)=\left(it-\frac{1}{2},t^2,t+i\right)$ is a complex parabola contained in $S$. One can check that $s(t)=i$ satisfies Eq. \eqref{asint2} and Eq. \eqref{geod2}; however, clearly $\bfy(t)$ does not define a straight line. The next result sheds some light on the kind of curves contained in $S$ which satisfy Eq. \eqref{asint2} and Eq. \eqref{geod2}, but are not straight lines.

\begin{lemma} \label{whoare}
Let $\bfy(t)=\bfx(t,s(t))$, with $s(t)$ a complex function, parameterize a complex curve ${\mathcal C}\subset S$. If $\bfy(t)$ satisfies  Eq. \eqref{asint2} and Eq. \eqref{geod2} but is not a straight line, then $\Vert \bfx_t \times \bfx_s \Vert^2$ identically vanishes over the points $(t,s(t))$.  
\end{lemma}

\begin{proof}  Eq. \eqref{asint2} and Eq. \eqref{geod2} are equivalent to conditions (1) and (2) in Lemma \ref{rectacom}. Since we are assuming that Eq. \eqref{asint2} and Eq. \eqref{geod2} hold, $\langle\bfy'',(\bfx_t\times \bfx_s)\times \bfy'\rangle=0$. Since $\langle\bfy'',(\bfx_t\times \bfx_s)\times \bfy'\rangle$ can be computed as a determinant consisting of the components of $\bfy''$, $\bfx_t\times \bfx_s$, $\bfy'$, we have $\bfy''=a(\bfx_t\times\bfx_s)+b\bfy'$ for some $a,b\in \mathbb{ C}$. On the other hand, since $\bfy''\in {\mathcal L}(\bfx_t,\bfx_s)$, and $\bfy'\in {\mathcal L}(\bfx_t,\bfx_s)$ too, we deduce that $\bfx_t\times \bfx_s\in {\mathcal L}(\bfx_t,\bfx_s)$. Hence, there exist $c,d\in \mathbb{ C}$ such that $\bfx_t\times \bfx_s=c\bfx_t+d\bfx_s$. Therefore, we have 
\begin{equation}\label{firstprod}
\langle \bfx_t\times \bfx_s,\bfx_t\times \bfx_s\rangle=c\langle \bfx_t,\bfx_t\times \bfx_s\rangle +d\langle \bfx_s,\bfx_t\times \bfx_s\rangle.
\end{equation}
Now $\langle \bfx_t,\bfx_t\times \bfx_s\rangle$ can be computed as a determinant whose rows are the components of $\bfx_t$, $\bfx_t$ and $\bfx_s$, respectively. Since the first and third rows of the determinant are equal, we get $\langle \bfx_t,\bfx_t\times \bfx_s\rangle=0$. Similarly, $\langle \bfx_s,\bfx_t\times \bfx_s\rangle=0$. So from Eq. \eqref{firstprod}, $\langle \bfx_t\times \bfx_s,\bfx_t\times \bfx_s\rangle=\Vert \bfx_t\times \bfx_s\Vert^2=0$.
\end{proof}

\begin{remark} The same condition, i.e. $\Vert \bfx_t \times \bfx_s \Vert^2=0$ is also obtained for curves $\bfx(t(s),s)$. Notice that in the complex case this condition does not mean that the point is singular, i.e. that $\bfx_t\times \bfx_s=0$. 
\end{remark}

If ${\mathcal C}$ is parameterized by $\bfy(t)=\bfx(t,s(t))$, with $s(t)$ a complex function, and $\bfy(t)$ satisfies Eq. \eqref{asint2} and Eq. \eqref{geod2} but is not a straight line, we say that it is a \emph{special curve}. Similarly for curves $\bfx(t(s),s)$, with $t(s)$ a complex function. Notice that since $\bfx(t,s)$ parametrizes a surface, $\Vert \bfx_t\times \bfx_s \Vert^2$ is not identically zero; therefore there are at most finitely many special curves contained in $S$. We will use this fact when proving Theorem \ref{ruled} in the next section.

\subsection{A detailed example.}\label{detail-ex}

Consider the surface $S$ parameterized by \[\bfx(t,s)=(-s^3+3st^2+3s,3s^2t-t^3+3t,3s^2-3t^2).\]This is a minimal surface of degree 9, called the {\it Enneper surface}. Eq. \eqref{as-rat} yields 
\begin{equation}\label{ej1eq}
(-54)(\omega-1)(\omega+1)(s^2+t^2+1)^4=0,
\end{equation}
while Eq. \eqref{diff3} yields 
\begin{equation}\label{ej2eq}
2(\omega^2+1)(-tw+s)=0.
\end{equation}
Hence $M(t,s,\omega)=\omega^2-1$, $N(t,s,\omega)=2(\omega^2+1)(-tw+s)$, and then 
\begin{equation}\label{ej3eq}
\xi(t,s)=16(-t+s)(t+s).
\end{equation}
Furthermore, $\delta_1(t,s)=s^2+t^2+1$ and $\delta_2(t,s)=(s^2+t^2+1)^2$. Additionally, $\delta_3(t,s)$ is the determinant of the second fundamental form, which yields $\delta_3(t,s)=-2916(s^2+t^2+1)^4$. Also, \[\Vert \bfx_t \times \bfx_s\Vert^2=81(t^2+s^2+1)^4,\]so $\delta_4(t,s)=(t^2+s^2+1)^4$. Therefore, 
\[\mu(t,s)=(t+s)\cdot (t-s)\cdot (s^2+t^2+1).\]
We observe that the curve $\mu(t,s)=0$ has three irreducible components over the complex. Let us analyze each one.
\begin{itemize}
\item [(1)] The component $t+s=0$ is obviously rational. Plugging $s=-t$ into $\bfx(t,s)$, we get 
\[\bfx(t,-t)=(2s^3+3s, -2s^3-3s, 0),\]which defines a straight line through $(0,0,0)$ parallel to the vector $(1,-1,0)$. However, in order to illustrate the criterion obtained from Proposition \ref{check-line}, let us see how to use this criterion here. Now let $\alpha(t,s)=t+s$, and let $a=1$, so that $m(b)=1+b$. Furthermore, 
\[\bfw(t,s)=(3s^2+6st-3t^2-3, 3s^2-6st-3t^2+3, 0).\]After evaluating $\bfw(t,s)$ at $t=1,s=-1$, we get 
\[\bfw_0=\bfw(1,-1)=(-9, 9, 0).\]Hence, 
\[\bfw(t,s)\times \bfw_0=(0,0,54s^2-54t^2).\]We observe that all the components of $\bfw(t,s)\times \bfw_0$ are divisible by $\alpha(t,s)=t+s$, so $\alpha(t,s)=0$ certainly corresponds to a straight line contained in $S$. 
\item [(2)] The component $t-s=0$ is again rational. Plugging $s=t$ into $\bfx(t,s)$, we get 
\[\bfx(t,t)=(2s^3+3s, 2s^3+3s, 0),\]which defines a straight line through $(0,0,0)$ parallel to the vector $(1,1,0)$. 
\item [(3)] The last component corresponds to $s^2+t^2+1=0$. The component is again rational, and can be parameterized by 
\[\left(\frac{\beta}{2}\left(z^2+\frac{1}{z}\right),\frac{z^2-1}{2z}\right),\]
where $\beta^2+1=0$. By plugging this parameterization into $\bfx(t,s)$, we get 
\[\left(-\frac{z^4-2z^2+1)(z^2-1)}{2z^3}, \frac{\beta (z^2+1)(z^4+2z^2+1)}{2z^3}, \frac{3(z^4+1)}{2z^2}\right),\]
 which clearly does not correspond to a straight line. However, let us check what the criterion from Proposition \ref{check-line} yields in this case. First, we have $\alpha(t,s)=t^2+s^2+1$; furthermore, we pick $t=1$, so that $m(b)=b^2+2$. Now 
\[\bfw(t,s)=(18s^2t-6t^3-6t, 6s^3-18st^2+6s, -24st).\]Furthermore, 
\[\bfw_0=\bfw(1,b)=(18b^2-12, 6b^3-12b, -24b).\]Therefore, 
\[
\begin{array}{rcl}
\bfw(t,s)\times \bfw_0&=&(144bs(s^2-3t^2+4t+1),-1152st-24b(18s^2t-6t^3-6t),\\
&&-288s^3+864t^2s-288s+24b(18s^2t-6t^3-6t)),
\end{array}
\]
and one can easily see that $t^2+s^2+1$ does not, for instance, divide the first component. So we deduce that $t^2+s^2+1=0$ does not give rise to any straight line of $S$. 
\end{itemize}

Finally, one can check that \[\eta(t,s)=\gcd\left(g^{\star}(t,s),\mbox{num}\left(\widehat{\Gamma}^1_{22}(t,s)\right)\right)=1,\] and therefore there are no straight lines of the type $\bfx({\bf c},s)$, where ${\bf c}$ is a constant. The whole computation takes 0.078 seconds. Figure 1 shows a picture of the surface, together with the two straight lines we have computed.

\begin{figure}
$$\begin{array}{c}
  \includegraphics[scale=0.3]{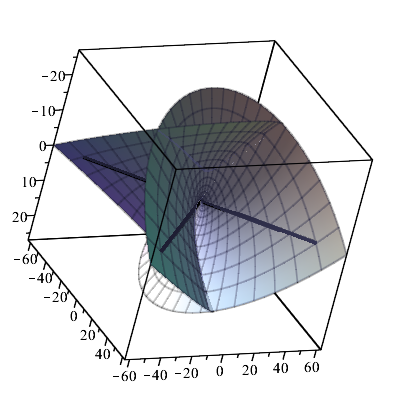}
 \end{array}$$
 \caption{Straight lines contained in an Enneper surface.}
\end{figure}

\section{How the algorithm recognizes ruled surfaces.}\label{term}

In this section, we will prove that polynomials $\mu(t,s)$ and $\eta(t,s)$ defined in steps 13 and 18 of the Algorithm {\tt STLines} identically vanish if and only if $S$ is a ruled surface. Since Algorithm {\tt STLines} requires $S$ not to be ruled, this proves that the loops in steps 14 and 19 do not run forever, so the algorithm works properly. Additionally, we observe that in fact we do not really need to check in advance whether or not $S$ is ruled: if while the algorithm is working we detect that either $\mu(t,s)$ or $\eta(t,s)$ identically vanishes, we give back the answer {\tt $S$ is a ruled surface}, which in particular means that $S$ contains infinitely many (maybe complex) straight lines. 

We need some previous results. The following lemma is perhaps known; however we could not find any appropriate reference in the literature, so we provide here a complete proof. 

\begin{lemma}\label{manylines}
Let $S$ be an algebraic surface. If $S$ contains infinitely many straight lines, then $S$ is a ruled surface.
\end{lemma}

\begin{proof} 
Let $X$ be the algebraic set of all the straight lines $\ell$ contained in $\overline{S}$, the Zariski closure of $S$, which is contained in the grassmannian $\mathbb{G}(1,3)$ of lines in the complex projective space $\mathbb{P}^3$. Furthermore, let $I=\{(x,\ell)\ |\ x\in \bar{S},\ell\in X\}$. Now consider the incidence diagram:

\[
\xymatrix{ & I \ar[ld]_p \ar[rd]^q& \\\overline{S} & & X}
\]

%\begin{equation}
%\begin{array}{rrcll}
 %& & I& &  \\
 %& p\swarrow & &\searrow q \\
%\overline{S} & & & & X
%\end{array}
%\end{equation}
Since $S$ contains infinitely many straight lines, we have dim$X\ge 1$. Additionally $p(I)$ is the union of the elements of $X$, considered as straight lines in $\mathbb{P}^3$. Since $p(I)$ is an infinite union of different straight lines it cannot be a curve, so $\mbox{dim}(p(I))\ge 2$. Therefore $p(I)$ is a closed surface contained in the irreducible surface $\overline{S}$. Hence $p(I)=\overline{S}$, so $p$ is surjective. Thus, for all $x\in S$ there is a straight line $\ell\in X$ passing through $x$ (equivalently, $\overline{S}$ is the union of its lines). Hence $\overline{S}$ is ruled, and so is $S$.
\end{proof}

Now we are ready to prove the result. We recall here that if $\mbox{det}(II)>0$ at a point $P_0\in S$, where $II$ is the matrix defining the second fundamental form of $\bfx(t,s)$, then we say that $P_0$ is an {\it elliptic point}; if $\mbox{det}(II)=0$, $P_0$ is a {\it parabolic point}; if $\mbox{det}(II)<0$, $P_0$ is a {\it hyperbolic point}.

\noindent\begin{theorem}\label{ruled}
A rational surface $S$ properly parameterized by $\bfx(t,s)$ is a ruled surface if and only some of the polynomials $\mu(t,s)$, $\eta(t,s)$ are identically zero.
\end{theorem}

\begin{proof}
$(\Rightarrow)$ If $S$ is ruled then $S$ contains infinitely many straight lines. If $\mu(t,s)$, $\eta(t,s)$ are not identically zero then they have finitely many irreducible components. Therefore, there should be infinitely many straight lines not covered by the parameterization. But this cannot happen, because the set of points of $S$ not covered by $\bfx(t,s)$ has dimension at most 1. 

\noindent $(\Leftarrow)$ If $\eta(t,s)$ is identically zero then by Lemma \ref{aux-str} any coordinate line $\bfx({\bf c},s)$, with ${\bf c}$ a {real} constant, is a straight line. Since we have one coordinate line of this type through every point, then $S$ is ruled. So assume that $\mu(t,s)$ is identically zero, in which case either $\xi(t,s)$ or $\delta_3(t,s)$ is identically zero. If $f^{\star},g^{\star}$ are both identically zero then $\delta_3(t,s)$ is a nonzero constant; otherwise $\delta_3(t,s)=\mbox{det}(II)$. In that case, if $\delta_3(t,s)$ is identically zero then $S$ is developable (see pg. 103 of \cite{Struik}), and therefore ruled. Therefore, let us assume that $\xi(t,s)$ is identically zero. 

Now let ${\mathcal M}$, ${\mathcal N}$ denote the algebraic surfaces in the $(t,s,\omega)$-space defined by $M(t,s,\omega)$ and $N(t,s,\omega)$. If $\xi(t,s)$ is identically zero, then ${\mathcal M}$, ${\mathcal N}$ share a component $P(t,s,\omega)=0$. We have three different cases, depending on whether $M$ has degree two, degree one (when $g^{\star}=0$) or degree zero (when $g^{\star}=f^{\star}=0$). We will address in detail the case when $M$ has degree two, which is the most difficult one. Later, we will make some observations on the other two cases.

Assuming that $M$ has degree two in $\omega$, and since $M$ is primitive with respect to the variable $\omega$ (because its content with respect to $\omega$ was removed), ${\mathcal M}$ has at most two irreducible components, corresponding to 
\[\omega=\frac{-f^{\star}+\sqrt{(f^{\star})^2-4e^{\star}g^{\star}}}{e^{\star}},\mbox{ }\omega=\frac{-f^{\star}-\sqrt{(f^{\star})^2-4e^{\star}g^{\star}}}{e^{\star}}.\]
Without loss of generality, we will assume that $P(t,s,\omega)=0$ contains the graph of 
\[\omega=\frac{-f^{\star}+\sqrt{(f^{\star})^2-4e^{\star}g^{\star}}}{e^{\star}}.\]

We have now two possibilities, depending on whether or not there exists an open set $\Omega\subset \mathbb{ R}^2$ where $(f^{\star})^2-e^{\star}g^{\star}\geq 0$. We focus on the affirmative case; the negative case will be addressed later. Let $\Omega\subset \mathbb{ R}^2$ be an open set where $(f^{\star})^2-e^{\star}g^{\star}\geq 0$, $e^{\star}\neq 0$, and all the Christoffel's symbols are well-defined. By Picard-Lindel\"of's Theorem, for any $(t_0,s_0)\in \Omega$ the initial value problem
\[
\left\{
\begin{array}{ccc}
\displaystyle{\frac{ds}{dt}} & = & \displaystyle\frac{-f^{\star}+\sqrt{(f^{\star})^2-4e^{\star}g^{\star}}}{e^{\star}},\\
s(t_0) & = & s_0
\end{array}
\right.
\]
has a unique, real, solution, $s(t)$, defined over an interval $I$ containing $t_0$. By construction, the space curve ${\mathcal C}_1$ (in the $(t,s,\omega)$-space) parameterized by 
\[\left(t,s(t),\frac{ds}{dt}(t)\right),\]with $t\in I\subset \mathbb{ R}$, $t_0\in I$, is contained in the surface $P(t,s,\omega)=0$. Now let $\omega_0=\frac{ds}{dt}(t_0)$. Also by Picard-Lindel\"of's Theorem, the initial value problem 
\[
\left\{
\begin{array}{ccc}
\displaystyle{\frac{dw}{dt}}&=&\psi(t,s(t),\omega),\\
\omega(t_0)&=&\omega_0
\end{array}
\right.
\]
where $\psi(t,s(t),\omega)$ represents the right hand-side of Eq. \eqref{diff2} evaluated at $s=s(t)$, has a unique, real, solution $\omega(t)$. Again, by construction the space curve ${\mathcal C}_2$ parameterized by 
\[\left(t,s(t),\omega(t)\right),\]with $t\in J$, $t_0\in J$, is contained in the surface $P(t,s,\omega)=0$. Now since ${\mathcal C}_1$, ${\mathcal C}_2$ are two analytic curves contained in the surface $P(t,s,\omega)=0$, sharing the point $(t_0,s_0,\omega_0)$, and projecting onto the same curve $(t,s(t))$ of the $(t,s)$-plane for $t\in I\cap J$, then ${\mathcal C}_1$ and ${\mathcal C}_2$ must coincide. Since $s(t)$ satisfies both Eq. \eqref{asint2} and Eq. \eqref{geod2} and $s(t)$ is real, by Theorem \ref{known} then $\bfx(t,s(t))$ must be a straight line through the point $P_0=\bfx(t_0,s_0)$. Since this construction is valid for $(t_0,s_0)\in \Omega$, we conclude that $S$ contains infinitely many straight lines. But then the implication follows from Lemma \ref{manylines}.

If $(f^{\star})^2-e^{\star}g^{\star}< 0$ for all $(t,s)\in \mathbb{ R}^2$, so that all the points of $S$ are elliptic, we proceed as before replacing Picard-Lindel\"of's Theorem by the Complex Existence and Uniqueness Theorem for differential equations in the complex domain (see for instance Theorem 2.2.1 in \cite{Hille}). This way, we construct infinitely many curves contained in $S$, satifying Eq. \eqref{asint2} and Eq. \eqref{geod2}. Since only finitely many of them can be special curves (see the end of Section \ref{subsec-complex}), we also get infinitely many straight lines (complex, this time) contained in $S$, so $S$ is ruled, and the result follows.

In the cases $g^{\star}=0,\mbox{ }f^{\star}\neq 0$ and $g^{\star}=f^{\star}=0$, we also invoke Picard-Lindel\"of's Theorem, applied this time to the differential equation $\omega=\frac{ds}{dt}=\Phi(t,s)$ provided by Eq. \eqref{as-rat}, when $g^{\star}=0,\mbox{ }f^{\star}\neq 0$, and tothe differential equation provided by the derivative with respect to $t$ of Eq. \eqref{as-rat}, when $g^{\star}=f^{\star}=0$. Notice that in these cases $\Phi(t,s)$ is, unlike the case when $M$ has degree two, a rational function, which simplifies the reasoning.
\end{proof}

\section{Experimentation.}\label{sec-experimentation}

In this section we report on our implementation of the algorithm {\tt STLines}, and we compare these timings with the timings corresponding to two brute-force approaches. All the examples have been run in an Intel Core computer, revving up at 2.90 GHz, with 8 Gb of RAM memory, using the computer algebra system {\tt Maple} 18. The code of our implementation can be freely downloaded from \cite{Jorge}; the parameterizations used in the experimentation can be found in Appendix I. 

\subsection{Brute-force approaches.}\label{brute-force}

If the implicit equation $F(x,y,z)=0$ of $S$ is available, a first brute-force approach to find the straight lines contained in $S$ consists of considering a generic line 
\[
x=t,\mbox{ }y=a t+b,\mbox{ }z=c t+d,
\]
and impose that $F(t,a t+b, ct+d)$ is identically zero. If the degree of $F$ is $N$, this amounts to solving a polynomial system ${\mathcal S}$ of degree $N$ in 4 variables; straight lines contained in planes normal to the $x$-axis are not computed this way, but they can be computed by a similar method, with fewer variables. The system ${\mathcal S}$ can be solved by using Gr\"obner bases. This is feasible even for serious degrees whenever $F$ is not dense, but it gets increasingly difficult in the dense case for $N>10$ (see later). 

If $S$ is defined by means of a rational parametrization, as it is the case in this paper, the cost is double, since we need to compute first the implicit equation, and then solve the system ${\mathcal S}$. The computation of the implicit equation of $S$ can be difficult, and in fact implicitization techniques are an on-going subject of research. A simple, general method for computing the implicit equation is to use Gr\"obner bases, jointly with Rabinowitsch's trick in the case when the parameterization has denominators. This is the method that we used in our experiments. However, there are other, often much better, more sophisticated techniques for computing the implicit equation, like $\mu$-bases \cite{Cheng, Deng, SR17-1, SR17-3}, moving planes \cite{SR17-2} or moving surfaces \cite{Zheng}. We did not use these other methods in our analysis, which in some cases can dramatically improve the computation time with respect to the Gr\"obner basis approach. However, even if the time devoted to computing the implicit equation can be improved, depending on the size of the implicit equation, we still have the problem of dealing with the system ${\mathcal S}$ (see Table 2).

The next table, Table 1, shows the performance of this method for some parameterizations, which can be found in \cite{Jorge}, next to the code for our algorithm {\tt STlines}. The columns correspond to the name of the example, the bidegree of the parameterization, the computation time (in seconds) to find the implicit equation (Time impl.), the degree $\mbox{deg}(F)$ of the implicit equation, the number of terms $\mbox{n}_{\mbox{terms}}$ of the implicit equation $F$, the computation time for finding the straight lines (Time solv.), i.e. to solve the system ${\mathcal S}$, and the total computation time $t_{\mbox{imp}}=\mbox{Time impl.}+\mbox{Time solv.}$ 

\begin{center}
\begin{tabular}{|c|c|c|c|c|c|c|c|} 
\hline Ex. & Bideg. & Time impl. & deg($F$) & $\mbox{n}_{\mbox{terms}}$ & Time solv. & $t_{\mbox{imp}}$ \\
\hline $S_1$ & (3,2)  & 0.499 & 9 & 23 &  0.234 &  0.733 \\
\hline $S_2$ & (2,2)  & 0.702 & 3 & 16 & 0.219 &  0.921 \\
\hline $S_6$ & (5,1)  & 347.898 & 10 & 62 & 0.141 & 348.039 \\
\hline $S_8$ & (5,2)  & 0.062 & 10 & 4 & 0.031 & 0.093 \\
\hline $S_9$ & (3,3)  & 2.433 & 14 & 8 & 0.094 & 2.527 \\
\hline $S_{11}$ & (3,5) & 0.047 & 16 & 2 & 0.015 & 0.062 \\
\hline $S_{15}$ & (2,2) & 0.125 & 6 & 5 & 0.047 & 0.172 
\\
\hline
\end{tabular}
\end{center}
\begin{center}
{\bf Table 1:} Examples for the brute-force approach computing the implicit equation. 
\end{center}

\vspace{0.2 cm}
In all these examples, corresponding to parameterizations which are not too complicated, the computation of the implicit equation takes little time, with the exception of $S_6$. However, once the implicit equation is there, and since in these cases the implicit equations, again, are not too big, the computation time of solving the polynomial system is very small. The case of $S_6$ is illustrative: even though computing the implicit equation takes time, solving the polynomial system does not. The effect of having larger implicit equations can be observed in the next table, Table 2. In this case we did not start from parameterizations, but from implicit equations $F$ randomly generated in the computer algebra system {\tt Maple}. One can see that as the degree and number of terms of $F$ grow larger, the computation time for solving the system gets much bigger. In Table 2 we show the degree deg($F$) of the implicit equation, the number of terms of $F$, a bound $\kappa$ on the absolute value of the coefficients, and the computation time of solving the polynomial system. As the implicit equation gets bigger, the computation turns unfeasible; nevertheless, if the equation is sparse, as in the last row of Table 2, the computation is still doable in reasonable time. 

\begin{center}
\begin{tabular}{|c|c|c|c|} 
\hline deg($F$) & $\mbox{n}_{\mbox{terms}}$ & $\kappa$ & Time solv. \\
\hline 8 & 155 & 5 & 17.191 \\
\hline 9 & 208  & 5 & 25.475 \\
\hline 10 & 52  & 10 & 68.734 \\
\hline 12 & 54 & 10 &  423.121 \\
\hline 15 & 71  & 10 & $>1000$ \\
\hline 20 & 14 & 25 & 17.285
\\
\hline
\end{tabular}
\end{center}
\begin{center}
{\bf Table 2:} Larger implicit equations
\end{center}

A second brute-force method, more suitable for a rational input, is the following: any straight line contained in $S$, not parallel to the $yz$-plane, is contained in the intersection of two planes 
\[
y=ax+b,\mbox{ }z=cx+d.
\]
Therefore, in order to compute the straight lines contained in $S$, we can consider the equations
\[
Q_1(t,s)=y(t,s)-a \cdot x(t,s)-b=0,\mbox{ }Q_2(t,s)=z(t,s)-c\cdot x(t,s)-d =0,
\]
where we can eliminate the variable $s$. This way we get a polynomial $P$ in the variable $t$, that must be identically zero. Again, this yields a polynomial system in 4 variables of degree ${\mathcal O}(N_1\cdot N_2)$, where $N_1,N_2$ are bounds for the degrees in the variables $t,s$ of the numerators of $Q_1,Q_2$. The straight lines parallel to the $yz$-plane can be computed in a similar way, solving a polynomial system with fewer variables. We compare, in Table 3, the timings between this method ($t_{\mbox{rat}}$) and the brute-force approach computing the implicit equation ($t_{\mbox{imp}}$): in all the cases, except for one case where both methods perform very well, this second brute-force approach beats the first one. The best time is shown in bold letter. Additionally, in Table 3 we also show the degree $\mbox{deg}(P)$ of the polynomial $P$.

\begin{center}
\begin{tabular}{|c|c|c|c|c|c|} 
\hline Ex. & Bideg.   & $t_{\mbox{imp}}$ & $\mbox{deg}(P)$ & $t_{\mbox{rat}}$ & $t_{\mbox{our}}$\\
\hline $S_1$  & (3,2)  & 0.733 & 9 & {\bf 0.109} & 0.218 \\
\hline $S_2$ & (2,2)  & 0.921 & 12 & {\bf 0.468} & 1.248\\
\hline $S_6$  & (5,1) & 348.039 & 3 & {\bf 0.032} & 1.404 \\
\hline $S_8$  & (5,2) & 0.093 & 10 & {\bf 0.0} & 0.094 \\
\hline $S_9$ & (3,3)  & 2.527 & 17 & {\bf 0.0} & 1.326 \\
\hline $S_{11}$ & (3,5) & {\bf 0.062} & 16 & 0.094 & 0.390 \\
\hline $S_{15}$ & (2,2) & 0.172 & 6 & {\bf 0.015} & 26.957 
\\
\hline
\end{tabular}
\end{center}
\begin{center}
{\bf Table 3:} Examples comparing both brute-force approaches
\end{center}

We have also added, in Table 3, a last column with the timing of our algorithm, {\tt STLines}: in all the examples in Table 3, where the parametrizations are not too complicated, we cannot produce better timings than the brute-force approaches. The situation changes, as we show in the next section, when we deal with bigger inputs.

\subsection{Our method}\label{subsec-timing}

In this subsection we test our algorithm against the brute-force approaches, with bigger inputs. The parameterizations of the surfaces $S_i$ can be found in \cite{Jorge}. Whenever we write $S_i^{\star}$, we mean a reparameterization of the surface with the following change of parameters:
\[
t:=\frac{2t}{t^2+s},\mbox{ }s:=\frac{3s}{t^2+s},
\]
which in general produces denser parameterizations, and with higher bidegree. The timings are provided in Table 4. Here we did not include the timings of the brute-force approach computing the implicit equation, but only the second brute-force approach; the reason is that in all the cases the computation time of the implicit equation exceeds reasonable timings. One can see that our method beats the brute-force approach in most of the cases; in fact, in most of the cases the brute-force approach cannot produce an answer in a reasonable amount of time. Additionally, in each case we show the number of straight lines found by our algorithm ($n_{\ell}$). The best time for each example is shown in bold letter. In the case of $S_{19}$, the asterisk denotes that Maple failed to compute the polynomial $P$ after a reasonable amount of time. 

\begin{center}
\begin{tabular}{|c|c|c|c|c|c|c|} 
\hline Ex. & Bideg.   & deg($P$) & $t_{\mbox{rat}}$ & $t_{\mbox{our}}$ & $n_{\ell}$ \\
\hline $S_1^{\star}$  & (6,3)  & 18 & $> 300$ & {\bf 0.817} & 2 \\
\hline $S_2^{\star}$ & (6,3)  & 16 & $> 300$ & {\bf 3.681} & 18 \\
\hline $S_4^{\star}$  & (6,3) & 28 &  $> 300$ & {\bf 10.155} & 1  \\
\hline $S_5^{\star}$  & (10,5) & 30 & $> 300$ & {\bf 0.749} & 1  \\
\hline $S_6^{\star}$ & (10,5)  & 54 & $> 300$ & {\bf 41.949} & 1  \\
\hline $S_7^{\star}$ & (8,4) & 32 & {\bf 5.194} & 7.192 & 1  \\
\hline $S_8^{\star}$ & (10,5) & 20 & {\bf 0.234} & 1.248 & 0  \\
\hline $S_9^{\star}$ & (6,3) & 44 & $> 300$ & {\bf 3.148} & 0  \\
\hline $S_{10}$ & (4,4) & 16 & $>300$ & {\bf 22.136} & 0  \\
\hline $S_{11}^{\star}$ & (12,6) & 70 & 1.357 & {\bf 0.406} & 0  \\
\hline $S_{12}$ & (2,3) & 8 & $>300$ &  {\bf 4.306} & 1  \\
\hline $S_{13}$ & (3,4) & 15 & $>300$ & {\bf 33.993} & 1  \\
\hline $S_{14}$ & (3,4) & 20 & {\bf 5.959} & 30.592 & 2  \\
\hline $S_{15}^{\star}$ & (8,4) & 16 & {\bf 0.125} & $>300$ & 2   \\
\hline $S_{17}$ & (8,8) & 64 & $>300$ & {\bf 1.747} & 0  \\
\hline $S_{18}$ & (10,10) & 100 & $>300$ & {\bf 2.309} & 0 \\
\hline $S_{19}$ & (9,9) & $\star$ & $>300$ & {\bf 19.313} & 0 \\
\hline $S_{20}$ & (7,7) & 98 & $>300$ & {\bf 8.970} & 2  \\
\hline $S_{21}$ & (12,12) & 144 & $>300$ & {\bf 9.407} & 1  \\
\hline $S_{22}$ & (13,13) & 169  &  $>300$ & {\bf 9.313} & 0  
\\
\hline
\end{tabular}
\end{center}
\begin{center}
{\bf Table 4:} Performance of our algorithm with bigger inputs
\end{center}

The surface $S_2^{\star}$ corresponds to the Clebscht surface, a non-singular cubic surface, so it is a good benchmark for the algorithm. In this case, the algorithm detects 18 straight lines because there are 6 lines not covered by the parameterization, and 3 lines at infinity; so $18+6+3=27$, as predicted by the classical theory. Although in most cases our algorithm beats the brute-force approach, in Table 4 we also see that there can be cases where the brute-force approach works better. This happens, in particular, in $S_{15}^{\star}$ (and also in $S_{15}$, as shown in Table 3). The reason behind the behavior of $S_{15}^{\star}$ is that the computation of the absolute factorization of $\mu(t,s)$ takes a long time. It is possible to reduce this computation by factoring $\xi(t,s)$ first over the rationals, and then computing the absolute factors of each irreducible factor of $\xi(t,s)$ over $\mathbb{ Q}$; this way, the computation time is reduced to 23.228 seconds. In any case, $S_{15}^{\star}$ is illustrative of the tradeoff between our method, and the second brute-force approach: while in our case one has to deal with absolute factoring, in the second brute-force approach one has to deal with polynomial systems in 4 variables. The latter can be advantageous when the input has small size, but it gets increasingly worse as the problem grows in size.

\section{Conclusions.} 

We have presented an algorithm to compute the straight lines contained in an algebraic surface, defined by a rational parameterization. The algorithm is based on the well-known result in Differential Geometry characterizing real, non-singular straight lines contained in a surface as lines which are simultaneously asymptotic lines, and geodesic lines. Experiments conducted on non-trivial examples with moderate to medium degrees, show that our method is generally better than the brute-force approach. 

The method can certainly be generalized to the case of implicit algebraic surfaces. In that case, we need to use the Implicit Function Theorem to write the coefficients of the first and second fundamental forms and the Christoffel symbols in terms of $x,y,z$. Afterwards, some auxiliary variables must be properly eliminated. However, the experiments we have conducted show that this approach is not better than the brute-force approach. For this reason, we have left the implicit case out of the paper.

\newpage

\section{Appendix I: parametrizations used in the examples.}

The rational parametrizations used in Section \ref{sec-experimentation} are the following:

\begin{itemize}
\item $S_1$:
\[
\bfx(t,s)=(-s^3+3t^2s+3s,3s^2t-t^3+3t,3s^2-3t^2)
\]
\item $S_2$:
\[
\bfx(t,s)=\left(\frac{(s-1)(ts+t-1)}{-t+t^2+s-s^2},\frac{-t^2s+t+s-1}{-t+t^2+s-s^2},\frac{t(1-t-s^2)}{-t+t^2+s-s^2}\right)
\]
\item $S_3$:
\[
\bfx(t,s)=(st,s^2(t-1),s^3(t+1))
\]
\item $S_4$:
\[
\bfx(t,s)=\left(t+s^3+t^3+1,\frac{2st+s+s^3+t^3+1}{s},\frac{3t^2-t+s^3+t^3+1}{t}\right)
\]
\item $S_6$:
\[
\bfx(t,s)=\left(t^3+\frac{s}{t^2+1}, \frac{t+s}{1+s}, t^5+s\right)
\]
\item $S_7$:
\[
\bfx(t,s)=(t^3-s, ts^3, s^4+t^3)
\]
\item $S_8$:
\[
\bfx(t,s)=(t, s^2, t^5+s)
\]
\item $S_9$:
\[
\bfx(t,s)=\left(\frac{s}{t^2},\frac{s^3+t^2}{s+t},t^3\right)
\]
\item $S_{10}$:
\[
\bfx(t,s)=\left(\frac{t^4+2s^3-st^2-2st}{-2s^4+2s^3+t^3+s^2}, \frac{-2s^4+2s^2t^2-2t^4+st^2-t}{-2s^4+2s^3+t^3+s^2}, \frac{-s^3t+st^3+2t^3+2s}{-2s^4+2s^3+t^3+s^2})\right)
\]
\item $S_{11}$:
\[
\bfx(t,s)=(t^3s^3, t^2s^4, s^5)
\]
\item $S_{13}$:
\[
\bfx(t,s)=\left(\frac{t(s^2-t^2-s)}{q(t)},\frac{s(-2t^3+2st-2t^2+s-1)}{q(t)},\frac{s(-2s^3-2s^2t-t^2+1)}{q(t)}\right),
\]
where $q(t)=-73s^4+97s^2t^2-62s^3-56s^2+87t$.
\item $S_{14}$:
\[
\bfx(t,s)=\left(\frac{t(-s^2t+2st^2+t^2+2s-t)}{q(t)},\frac{s(s^3t+2st^3-2s^3+2st^2)}{q(t)},\frac{s(-2s^3t-2s^2t^2-st)}{q(t)}\right),
\]
where $q(t)=-10s^4-83s^2t^2-4st^3-73s^2+97t^2-62t$.
\item $S_{15}$:
\[
\bfx(t,s)=(t, t^2(s^2+1), s^2+s+1).
\]
\item $S_{17}$:
\[
\bfx(t,s)=(t^8, s^8, -10s^4-83s^2t^2-4st^3-73s^2+97t^2-62t)
\]
\item $S_{18}$:
\[
\bfx(t,s)=(t^{10}, s^{10}, -10s^4-83s^2t^2-4st^3-73s^2+97t^2-62t)
\]
\item $S_{19}$:
\[
\bfx(t,s)=(t^9, s^9, -10s^4-83s^2t^2-4st^3-73s^2+97t^2-62t)
\]
\item $S_{20}$:
\[
\bfx(t,s)=(t^7, s^7, -82s^7+62s^5t^2-10s^3t^4-83t^7-4s^2-73st)
\]
\item $S_{21}$:
\[
\bfx(t,s)=(t^{12}, s^{12}, -82s^2t^9+62s^3t^7-10s^8-83s^7t-4st^7-73s^2)
\]
\item $S_{22}$:
\[
\bfx(t,s)=(t^{13}, s^{13}, -82s^7+62s^5t^2-10s^3t^4-83t^7-4s^2-73st)
\]
\end{itemize}

\end{document}